\documentclass[psamsfonts,11pt]{amsart}
\usepackage{amsmath}
\usepackage{amssymb}
\usepackage{amsthm}

\usepackage{graphicx, epsfig, psfrag}

\def \T{\mathbb T}
\def \C {\mathbb C}
\def \Q {\mathbb Q}
\def\ep {\varepsilon}

\def \Z{\mathbb Z}
\def \Zl{\mathbb Z^\ell}
\def \Zk{\mathbb Z^k}

\def \Znn{\mathbb Z^{n-1}}
\def \Z2{\mathbb Z^2}
\def \Z3{\mathbb Z^3}
\def \Za{\mathbb Z^a}

\def \Rk {\mathbb R^k}
\def \Rl {\mathbb R^\ell}
\def \Rnn {\mathbb R^{n-1}}

\def \Rn{\mathbb R^n}
\def \R2{\mathbb R^2}
\def \R3{\mathbb R^3}
\def \Rb {\mathbb R^b}

\def \Tn{\mathbb T^n}

\def \T2{\mathbb T^2}
\def \T3{\mathbb T^3}

\def \C2{\mathbb C^2}
\def \C3{\mathbb C^3}

\def \S1{\mathbb S^1}
\def \S2{\mathbb S^2}
\def \S3{\mathbb S^3}

\def\a{\alpha}
\def\ao{\alpha_0}

\def\g{\gamma}

\def\ve{\varepsilon}

\def\la{\lambda}

\def\x{\xi}

\def\S{\Sigma}

\def\0C{\mathcal C^0}
\def\1C{\mathcal C^1}
\def\2C{\mathcal C^2}

\def\E2P {\exp 2\pi i}

\newcommand{\Id}{\operatorname{Id}}

\newcounter{Lemma}

\newtheorem{lem}{Lemma}[section]
\newtheorem{cor}{Corollary}[section]
\newtheorem{prop}{Proposition}[section]
\newtheorem{thm}{Theorem}[section]

\newtheorem{conj}{Conjecture}
\newtheorem*{theorem*}{Theorem A}

\theoremstyle{definition}

\newtheorem{rem}{Remark}%[section]
\newtheorem{defn}{Definition}%[section]

\numberwithin{equation}{section}

\begin{document}
\title[The Fried entropy and slow entropy]
{The Fried average entropy and slow entropy for actions of higher rank abelian groups}
\author[A. Katok, S. Katok and F. Rodriguez Hertz]{ Anatole Katok $^*$ }
\author[]{Svetlana Katok}
\author[]{Federico Rodriguez Hertz $^{**}$}

\subjclass[2010]{}

\address{Department of Mathematics\\
        The Pennsylvania State University\\
        University Park, PA 16802 \\
        USA}
\email{katok\_a@math.psu.edu}

\email{katok\_s@math.psu.edu}

\email{hertz@math.psu.edu}
\thanks{ $^*$ Based on research  supported by NSF grants DMS 1002554 and 1304830.}
\thanks{ $^{**}$ Based on research  supported by NSF grant DMS 1201326.}

\begin{abstract} We consider two numerical entropy-type invariants  for actions of $\Zk$, invariant under  a choice of generators and  well-adapted  for  smooth actions whose individual elements have positive entropy. We concentrate on the maximal rank case, i.e. $\Zk,\,k\ge 2$ actions on $k+1$-dimensional manifolds. In this case we show that for a fixed dimension (or, equivalently,  rank)  each of the invariants determines the other and their values are closely related to regulators in algebraic number fields. In particular, in contrast with the classical case of $\mathbb Z$ actions,  the  entropies of ergodic maximal  rank actions take only countably many values.   Our main result is the dichotomy that is best expressed under the assumption of weak mixing or, equivalently, no periodic factors: either both invariants vanish, or their values are bounded away from zero by  universal constants. Furthermore, the lower bounds  grow with dimension:  for the first invariant (the Fried average entropy) exponentially,  and for the second (the slow entropy)  linearly.
\end{abstract}

\maketitle

\section{Introduction}
It is well-known that the standard  notion  of entropy for an action of a locally compact compactly generated  topological group $G$ 
by measure-preserving transformations  assigns  zero value  to any smooth action unless the group $G$ is virtually cyclic, i.e. a compact extension of $\mathbb Z$ (notice that, in particular,  $\mathbb R$ is a compact extension of $\mathbb Z$) .\footnote{Condition of compactly generated (finitely generated in the case of discrete groups) is essential. For example,  {\em any} action of the direct limit of finite cyclic groups allows smooth realization (the first author, unpublished).} The reason is very simple: entropy of a group action measures in some way the exponential  growth of the number of  distinguishable orbit segments  against the  size or the volume  of, say,  a ball in the word-length metric  or some left-invariant metric  in the acting group while for the smooth (or Lipschitz) actions  this number  growth no faster than exponentially against  the radius of that ball. 

There are two  natural ways to obtain numerical invariants  reflecting orbit growth that  are non-trivial for smooth measure-preserving actions:
\begin{enumerate} 
\item  Change  the normalization  and  measure exponential growth against  the radius of the ball. This leads to the slow entropy type invariants as in \cite{KTh}.
\item  Consider an  average exponential  growth, represented by the  ordinary entropy, along one-parameter subgroups, or, more generally,  ``directions'' in the group; for higher-rank abelian groups  this was done by D. Fried  \cite{DF}. 
\end{enumerate}

In this paper we only consider actions of higher rank abelian groups,  more specifically, $\Zk \times\Rn$. In fact, for the general discussion it is sufficient to consider actions of $\Rk$. A standard suspension construction allows to reduce actions of $\Zk$, or, more generally, $\Zk \times\Rn$ to this case. Let us fix a volume element in $\Rk$; in the suspension case a natural volume element comes from $\Zk$.
For $\Rk$ actions the  group of automorphisms  preserving the volume element  and orientation is naturally $SL(k, \mathbb R)$; for $\Zk$ actions the group automorphisms is $SL(k, \mathbb Z)$.  If the orientation is not fixed,  one takes corresponding two-point extensions of those groups.

A   notion of entropy for $\Rk$-actions in order to be  useful in the smooth case should be:
\begin{enumerate}
\item[(i)]  invariant under volume preserving automorphisms of $\Rk$; \footnote{In the suspension case this amounts to invariance under a choice of generators in $\Zk$.}
\item[(ii)] equal to the usual entropy for $k=1$;
\item[(iii)] positive   if entropy of all non-identity elements of the action is positive;
\item[(iv)]  finite for smooth actions on compact manifolds with respect to any Borel probability invariant measure.
\end{enumerate}
 
D. Fried  \cite{DF} suggested a notion of entropy for actions of higher rank abelian groups satisfying these properties, based on the averaging  approach (2). We discuss this notion in section~\ref{SectionFried}. For $\Zk$ actions   one can also define  entropy type invariant for  any rank between $1$ and $k$ based on Fried's approach, see Definition~\ref{k-entropy}.   

A notion based on the slow entropy approach (1) is described in section~\ref{SectionSlow}. While in the Fried definition the invariance property (i) is automatic, in the slow entropy approach 
it is only achieved by minimization over a class of norms in the acting group that of course makes  calculations  more difficult. 

Principal parts of this paper,  sections~\ref{SectionMain} and \ref{sbsSlowCartan}, deal with  a special class of actions, namely actions of $\Znn$  by hyperbolic automorphisms of the torus $\Tn$, $n\ge 3$ (Cartan actions). 

 In section~\ref{SectionMain} we use sophisticated  number-theoretic tools to show that the Fried  average entropy (Definition~\ref{FriedEntropy}) of such actions is bounded  from below by a positive constant, obtain a good below estimate for that constant, and  calculate the minimal value that is conjecturally  achieved   for   a certain $\mathbb Z^3$ action on $\mathbb T^4$; see  the first line of Table 2. Furthermore,  the lower bound grows exponentially with the rank of  a Cartan action, see Theorem~\ref{Fried-entropy-bound}.  We also obtain lower bounds (polynomial and exponential) for  many (but not all) intermediate entropies (Propositions~\ref{l-ent} and \ref{mahler-exp}). An exponential bound for the remaining case is formulated as Conjecture~\ref{conj-exp}.  
  
  The reason for particular attention we pay to Cartan actions is that  they and their minor  modifications provide the only models of $\Znn$ actions with positive Fried average entropy  on $n$-dimensional manifolds for $n\ge 3$. This is proved in \cite{KRH-a} and is quoted  below as Theorem~A.  Thus our estimates of the Fried 
average entropy and intermediate entropies are universal for weakly mixing actions of $\Znn$ on $n$-dimensional manifolds, $n\ge 3$ with positive Fried average entropy, see Corollary~\ref{CorollaryMain}. 

In section~\ref{sbsSlowCartan} we calculate our  automorphism invariant slow entropy for Cartan actions. Our main conclusion (Theorem~\ref{slow-entropy-bounds}) is that  it is uniquely determined by the dimension and the value of the Fried average entropy.  Moreover, it is uniformly bounded away from zero. Again the measure rigidity Theorem~A implies that the same is true for any weakly mixing action of $\Znn$  with positive slow entropy   on an  $n$-dimensional manifold for $n\ge 3$.

Dropping the weak mixing assumption that in our situation is equivalent to the absence of periodic  factors, allows to construct examples with arbitrary small but positive  Fried average entropy  and slow entropy, albeit on manifolds of increasing complexity \cite[Remark 7]{KRH-a}. 
\subsection*{Acknowledgments}The estimates in section~\ref{SectionMain} are based on lower bounds for regulators in number fields.  We would like to thank Peter Sarnak for pointing out the original reference \cite{S} to us, and to John Voight and Eduardo Friedman for help and encouragement of the work on the identification of the number field that minimizes the Fried average entropy.

\section{Entropy function and the Fried average  entropy} \label{SectionFried}  
Let $\a$ be  an  action of $\Rk$ preserving a probability ergodic measure $\mu$. We will only consider actions all of whose elements have finite entropy as in the smooth case. 
\begin{defn}
The  {\em entropy function} for the action $\a$, denoted by $h^\a_\mu$, associates to ${\bf t}=(t_1,\dots t_k)\in \Rk$ the value of measure-theoretic entropy of $\a(\bf t)$, i.e. $h^\a_\mu({\bf t})=h_\mu(\a(\bf t))$. 
\end{defn}
For general actions  the  entropy function is not known to possess any nice properties other than the obvious positive homogeneity of degree one and central symmetry; see \cite[Appendix A(b)]{OW}   for a specific pathological example. However  for smooth actions  this function is convex and   piece-wise linear \cite{HHu}. Lyapunov exponents of the action $\a$ for such  
measure $\mu$ are linear functionals $\chi_j:\Rk\to\mathbb R,1\le j\le n$ , and their
kernels  are called the {\em Lyapunov hyperplanes}. We will routinely ignore $k$ zero  Lyapunov exponents that come from the orbit directions of the action.  Connected components of the complement to the union of the Lyapunov hyperplanes are called {\em Weyl chambers}. For any $C^{1+\epsilon}$ action on a compact manifold, or, more generally, for an action on a smooth manifold which is uniformly $C^{1+\epsilon}$, the entropy function is linear inside each Weyl chamber.\footnote{There are other cases of piecewise linearity that, however, require dynamical assumptions; see e.g. \cite[Theorem 6.16]{BL} where the role of Weyl chambers is played by the cones of expansive directions in a continuous $\Zk$ action.}  In general, maximal cones of linearity 
may comprise several Weyl chambers. If maximal cones of linearity coincide with Weyl chambers we 
say that the  action satisfies {\em full entropy condition} with respect to the  measure $\mu$. 
Full entropy condition is always satisfied if the measure $\mu$ is absolutely continuous. In this case, by the Pesin entropy formula, 
\begin{equation}\label{entropyformula}
 h^\a_\mu({\bf t})=h_\mu(\a({\bf t}))=\sum_{j:\chi_j>0}\chi_j({\bf t}), \text{ for }{\bf t}=(t_1,\dots,t_k)\in\Rk.
 \end{equation}
Since the measure $\mu$ is absolutely continuous, we have $\sum_{j=1}^n \chi_j=0$, therefore
(\ref{entropyformula}) can be re-written as
\begin{equation}\label{ell1}
h^\a_\mu({\bf t})=h_\mu(\a({\bf t}))=\frac12\sum_{j=1}^n|\chi_j({\bf t})|.
\end{equation}

Now we return to the general case of an ergodic Borel probability measure. There are two possibilities:
\medskip

\begin{enumerate}
\item[({\bf P})] {\em all elements of the action, save for the identity, have positive entropy;}  
 \item[({\bf O})] {\em some non-trivial element of the action has zero entropy.} 
 \end{enumerate}\medskip
 
In the first case  the entropy function defines a norm in $\Rk$; the unit ball is a polyhedron whose faces are sections of maximal cones of linearity. In the second case  the entropy function  is a semi-norm. 

\begin{defn}\label{FriedEntropy}{\em The Fried average entropy} of an action $\a$ of $\Rk$  with a fixed volume element, denoted by $h^*(\a)$, is the inverse of the volume of the unit ball $B(h^\a_\mu)$ in the entropy function norm/semi-norm multiplied by the volume of the generalized octahedron in $\Rk$, (the unit ball in the $\ell^1$ norm) $\mathcal O=\{(x_i)\in\Rk | \sum_{i=1}^k|x_i|\leq 1\}$, which is equal to
$\frac{2^k}{k!}$. Thus $$h^*(\a)=\frac{2^k}{k!vol(B(h^\a_\mu))} .$$ 

The Fried average entropy   of a $\Zk$ action $\a$ is defined as the Fried average entropy of its suspension. 

For a mixed group $\Za\times\Rb$ with  $a+b=k$ one takes suspensions for the discrete generators; as before the volume element in $\Rb$ is assumed to be fixed. 
\end{defn}

\begin{rem} Fried was aware of the polyhedral structure of the entropy function for absolutely continuous measures, but not in the general smooth case since the latter relies on the Ledrappier-Young generalization of the Pesin entropy formula \cite{LY} that appeared later than Fried's work, and was explicitly established  even later by Hu \cite{HHu}.
\end{rem}

The normalization is chosen in such a way that the following  properties hold:

\begin{prop}\*
 \begin{enumerate}
\item Let $\a_1$ be an action of  $\mathbb R^{k_1}$ on $M_1$, $\a_2$ be an action of $\mathbb R^{k_2}$ on $M_2$, preserving the measures $\mu_1$ and $\mu_2$, correspondingly, and $\a=\a_1\times\a_2$ be the product action of $\mathbb R^{k_1+k_2}$
 with the measure $\mu=\mu_1\times\mu_2$. Then $h^*(\a)=h^*(\a_1)h^*(\a_2)$.
\item For $k=1,\,\, h^*(\a)$ is equal to the usual entropy of the flow, i.e. the entropy of its time-one map. 
\end{enumerate}
\end{prop}
\begin{proof} (1)   Let $B(h^{\a_1}_{\mu_1})$ and $B(h^{\a_2}_{\mu_2})$ be the unit balls in the entropy norms for $\a_1$ and $\a_2$, respectively, and $B(h_\mu^{\a})$ - the unit ball in the entropy norm for $\a$. We use the additivity of entropy for Cartesian products and  the following property of the volume in $\ell^1$-norm on a product of Euclidean spaces provided with norms:
\[
(k_1+k_2)!vol(B(h^\a_\mu))=k_1!vol(B(h^{\a_1}_{\mu_1}))k_2!vol(B(h^{\a_2}_{\mu_2})).
\]
Then 
\[
\begin{aligned}
h^*(\a)=\frac{2^{(k_1+k_2)}}{vol(B(h^\a_\mu))(k_1+k_2)!}=&\frac{2^{k_1}}{k_1!vol(B(h^{\a_1}_{\mu_1}))}\frac{2^{k_2}}{k_2!vol(B(h^{\a_2}_{\mu_2}))}\\
=&h^*(\a_1)h^*(\a_2).
\end{aligned}
\]
(2) For $k=1$, the volume of the unit ball in the $\ell^1$-norm in $\mathbb R$, $\mathcal O=[-1,1]$
 is equal to $2$. The entropy function is given by $h(t)=th_\mu((\a(1))$. Hence $h_\mu(\a)=h_\mu(\a(1))=\frac{2}{vol(B(h_\a))}=h^*(\a)$.\end{proof}

The following properties of the Fried average entropy follow immediately from the definition: 
\begin{prop}\*\label{Fried-properties}
\begin{enumerate}
 \item The Fried average entropy is equal to zero exactly in case {\rm{(\bf O})}. 
\item The Fried average entropy of a $\Zk$ action 
 is independent of the choice of generators. 
\item\label{entropy-index} The Fried average entropy of the restriction of a $\Zk$ action $\a$ to a finite index subgroup  $A\subset \Zk$ is equal to $h_\mu^\a$ multiplied by the index of  the subgroup $A$. 
\end{enumerate}
\end{prop}

\begin{defn}\label{k-entropy} For $1\le l\le k$ the {\em $l$-entropy} of  a $\Zk$ action $\alpha$, denoted by $h^\ell_\mu$, is the infimum  of Fried average entropies  of ${\mathbb Z}^l$ subactions $\beta$ of $\alpha$:
\[
h^\ell(\a)=\inf_{\beta\subset\a} h^*(\beta).
\]
\end{defn}

In particular, $k$-entropy of a $\Zk$ action is  its Fried average entropy. 

\section{Algebraic actions and rigidity of the Fried average entropy for maximal rank actions.}\label{SectionMain}
\subsection{Maximal rank actions}
We will call an action of $\Za\times\Rb$ with  $a+b=n-1$ on a $n+b$--dimensional manifold a {\em  maximal rank action}. Any  action  whose rank is higher than  that necessarily has zero Fried average entropy since in this case  intersection of all Lyapunov hyperplanes    contains a non-zero element and  hence    suspension of the action   has   an element  all of whose Lyapunov exponents are non-positive and hence has zero entropy.

We will say that  $\a$  is a {\em general position action}  if the   the Lyapunov hyperplanes,  are in general position, i.e. the intersection of any number of those has the minimal possible dimension. 

Notice that if $\a$ is a general position maximal rank $\Zk$ action with $k=n-1$, the number of Weyl chambers (where all Lyapunov exponents have the same signs) is equal to $2^n-2$. Therefore either
\begin{enumerate}
\item there is a Weyl chamber where all signs are negative, or 
\item
all sign combinations excepts for all pluses and all minuses appear in some Weyl chambers. 
\end{enumerate}
In the  case  (1)  the measure has to be atomic \cite{K-DCDS}, hence the entropy function is identically zero. A simple argument shows that the general position assumption can be removed.

\begin{prop}\label{general-position} If $\a$ is not a general position  maximal rank action then its Fried average entropy is equal to zero. 
\end{prop} 
\begin{proof} Let $m$ be the  maximal number of Lyapunov hyperplanes whose intersection has dimension greater than minimal, i.e. greater than  $n-1-m$. Consider the restriction of the action to the intersection of those $m$ Lyapunov  hyperplanes that we denote by $L$. Intersections of remaining  $n-m$  hyperplanes with  $L$ are in general position; hence they divide $L$ into $2^{n-m}$ domains where the exponents have all possible combinations of signs. In particular there is a domain $D$ where all  $n-m$ exponents are negative. Since the remaining exponents vanish on  $L$ all action elements from $D$ have zero entropy. 
\end{proof}

So we only consider the case (2). In that case, unless the entropy function is identically zero, the measure $\mu$ is absolutely continuous \cite{KKRH}. 
To summarize, the (deep) results of \cite{KKRH} and the above arguments provide  the following   dichotomy  for ergodic invariant measures for maximal rank  actions  of  rank  $k\ge 2$: 
\medskip

\noindent ${\bf (\frak D)}${\em  Either  the measure is absolutely continuous and the Fried average entropy is positive, or  Fried average entropy of the action is equal to zero. }
\subsection{Cartan actions}
Examples of $\Znn$ actions with positive Fried average entropy (and positive $l$-entropies for all 
$l,\,1\le l\le n-1$) are Cartan actions  by automorphisms or affine maps  of the torus $\mathbb T^{n}$  and its finite factors called the infratori (see e.g.  \cite[Section 2.1.4]{KN}) with respect to Lebesgue measure.

The entropy function for  those  actions is described in terms of  groups of units in the rings of integers in algebraic number fields as follows, see \cite{KKS}. 
Let $A\in SL(n, \mathbb Z)$ be a matrix with an irreducible characteristic polynomial and hence distinct eigenvalues. The centralizer of  $A$ in the ring 
$M(n,\Q)$ of $n\times n$ matrices with rational coefficients can be identified with the ring of all polynomials in
$A$ with rational coefficients modulo the principal ideal
generated by the polynomial $f(A)$, and hence with the field
$K=\Q(\la)$, where
$\la$ is an eigenvalue of $A$. This identification is given  by an injective map
\begin{equation}\label{eq:gamma}
\g: p(A)\mapsto p(\la)
\end{equation}
with $p\in\Q[x]$. Let $\mathcal O_K$ denote the ring of integers in $K$ and $\mathcal U_K$ the group of units in $\mathcal O_K$.
Notice that if $B=p(A)\in M(n,\mathbb Z)$  then $\g(B)\in \mathcal O_K$, and
if $B\in GL(n,\mathbb Z)$ then $\g(B)\in \mathcal U_K$ 
(converse is not
necessarily true). We denote the centralizer of $A$ in   $M(n,\mathbb Z)$ by $C(A)$, and  the centralizer of $A$ in the group $GL(n,\mathbb Z)$ by $Z(A)$. 
It was shown in \cite{KKS} that $\g(C(A))$ is an order in $K$,
\begin{equation}\label{3.6}
\mathbb Z[\la]\subset \g(C(A))\subset \mathcal O_K,
\end{equation}
and
$Z(A)$ is isomorphic to the group of units in $\g(C(A))$.
Therefore, by the Dirichlet Unit Theorem \cite{BS}, %, Ch.2, \S 4.3),
to $\mathbb Z^{r_1+r_2-1}\times F$,  where $r_1$ is the number of the real
embeddings, $r_2$ is the number of pairs of complex conjugate
embeddings of the field
$K$ into $\mathbb C$, and $F$ is a finite cyclic group. We have $r_1+2r_2=n$, and
for the Cartan action, i.e. when $r_1+r_2-1=n-1$,
we conclude that $r_2=0$, i.e. $K$ is a totally real number field, and $F=\{\pm 1\}$.

Let $A_1=A,A_2,\dots , A_{n-1}\in Z(A)$ correspond to the multiplicatively independent units of $K$, $\ep_1,\dots , \ep_{n-1}$. More precisely, let $v=(v_1,\dots,v_n)$ be an eigenvector of $A$ with eigenvalue $\la$ whose coordinates belong to $K$. By (\ref{eq:gamma}) $\gamma(A)=\la$ is a unit of $K$, call it $\ep_1$. We have $A_i=p_i(A)$, hence $A_iv=p_i(\la)v$,
where $p_i(\la)$ is another unit of $K$, $\ep_i$.
Let $\varphi_1=id, \varphi_2,\cdots,\varphi_n$ be different embedding of $K$ into $\mathbb R$.
Then by taking different embeddings of the relation
$
A_iv=\ep_iv$, we obtain $A_i\varphi_j(v)=\varphi_j(\ep_i)\varphi_j(v)$, i.e. the eigenvalues of $A_i$ are units conjugate to $\ep_i$, and we obtain the following result.

\begin{prop} \label{diag} $A_i$ is conjugate to
\[
\begin{pmatrix} \ep_i & 0 & \cdots 0\\
0 & \varphi_2(\ep_i) & \cdots & 0\\
\cdots & \cdots & \cdots & \cdots\\
0 & 0 & \cdots & \varphi_n(\ep_i)
\end{pmatrix}
\]
over $\mathbb R$.
\end{prop}
Any Cartan action $\a$ by commuting automorphisms $A_1,\dots, A_{n-1}$ of the torus $\mathbb T^n$ is described by the above construction.

Now we will discuss the $\ell$-entropies of Cartan actions for fixed values of $\ell$, and the Fried average entropy. 

The $1$-entropy is simply the minimal value of the entropy of a non-identity element of the action. It is equal to the logarithm of the Mahler measure of the corresponding unit in the field $K$. By \cite{Sch, HS} for totally real fields the Mahler measure of any unit grows exponentially with the exponent bounded from below by 

\begin{equation}\label{Mahlerreal}c=\frac{1}{2}\log\left(\frac{1+\sqrt 5}{2}\right)\end{equation}  Thus we have the following statement:
\begin{prop} \label{1ent}The $1$-entropy of a Cartan action of a given rank $n-1$ is bounded from below by the linear function $cn$.
\end{prop}
Adding some simple geometric considerations, we obtain a more general result:
\begin{prop}\label{l-ent} The $\ell$-entropy of a Cartan action of  rank $n-1$ is bounded from below by  $\frac{c^\ell n^\ell}{\ell!}$.
\end{prop}
\begin{proof} Consider a $\Zl$ subaction $\beta$ of a Cartan action $\a$. We need to estimate the Fried average entropy of $\beta$, $h^*(\beta)$.
The entropy norm $h_\mu^\a$ restricts to $\Rl$ as $h_\mu^\beta$, and by Proposition \ref{1ent} for each element of the lattice ${\bf k}\in\Zl$ $h_\mu^\beta({\bf k})>cn$. Consider a ball in the entropy norm $h_\mu^\beta$ of radius $cn$. Since it contains no points of the lattice $\Zl$ except for the origin, by Minkowski Theorem, its volume is less then $2^\ell$. Thus
\[
c^\ell n^\ell vol(B(h_\mu^\beta))<2^\ell,
\]
hence
\[
vol(B(h_\mu^\beta))<\frac{2^\ell}{c^\ell n^\ell}.
\]
Therefore
\[
h^*(\beta)>\frac{c^\ell n^\ell}{\ell!}. 
\]
\end{proof}
Using Stirling formula we see that Proposition~\ref{l-ent} implies exponential below estimate for the $\ell$-entropies for  $\ell$ growing proportionally to the rank with a sufficiently small coefficient of proportionality. 

\begin{prop}\label{mahler-exp}  Given $\delta >0$, if $\delta< \frac{\ell}{n}< ec=\frac{e}{2}\log(\frac{1+\sqrt 5}{2})$, the $\ell$-entropy  of a Cartan action  of rank $n-1$ grows exponentially with $n$; the exponent is bounded below by  the positive number $\frac{\ell}{n}\log c-\frac{\ell}{n}\log\frac{\ell}{n}+\frac{\ell}{n}$.
\end{prop}

The estimate for the Mahler measure \eqref{Mahlerreal} is not sufficient for  estimating the Fried average entropy for Cartan actions. This estimate requires more sophisticated number-theoretic tools.

 Consider the entropy norm in $\Rnn$  associated  with a Cartan action $\a$ given by commuting matrices $A_i, 1\leq i\leq n-1$ with real eigenvalues  $\lambda_j(A_i), 1\leq j\leq n$. We define $n$
Lyapunov hyperplanes on $\Rnn$ by $\chi_j({\bf t})=0,\,\,j=1,\dots,n$, where
\begin{equation}\label{Lyap}
\chi_j({\bf t})=\chi_j(t_1,\dots t_{n-1})=\sum_{i=1}^{n-1} t_i\log|\lambda_j(A_i)|
\end{equation}
for $1\leq j\leq n$. In this case all hyperplanes are distinct. Since 
\begin{equation}\label{sumcolumnzero}
\sum_{j=1}^{n}\log|\lambda_j(A_i)|=0,
\end{equation}
 we have
\begin{equation}\label{sumzero}
\chi_1+\chi_2+\cdots \chi_{n}=0,
\end{equation}
hence inside each Weyl chamber the signs of $\chi_j$'s cannot be all the same, and the Lyapunov hyperplanes  divide $\Rnn$ into 
$2^{n}-2$ Weyl chambers, in which the entropy function $h^\a_\mu$ is  linear  Hence
the entropy function is a norm (see Figure 1 for $n=3$).
\begin{figure}[hbt]
 \psfrag{A}[l]{\small $ \chi_3=0$}
  \psfrag{B}[l]{\small $\chi_2=0$}
   \psfrag{C}[l]{\small $\chi_1=0$}
   \psfrag{D}[l]{\small $\chi_3=1$}
\psfrag{E}[l]{\small $\chi_3=-1$}
    \psfrag{C1}[l]{\small $\mathcal C_1$}
     \psfrag{C2}[l]{\small $\mathcal C_2$}
      \psfrag{C3}[l]{\small $\mathcal C_3$}
  \includegraphics[scale=0.5]{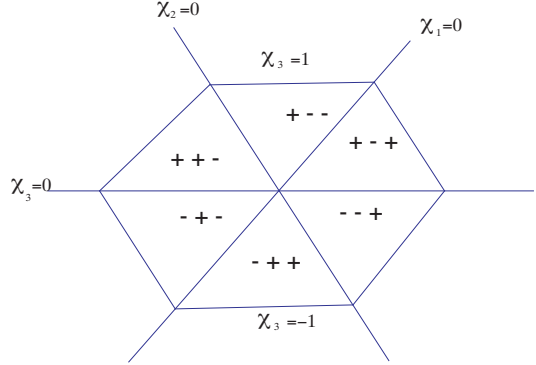}  
  \caption{Weyl chambers and the unit ball in the entropy function norm for $n=3$}
   \end{figure}
   
In order to calculate the volume of $B(h^\alpha_\mu)$, we make change of variables $\varphi:\Rn\to\Rn$
\[
y_j=\chi_j(t_1,\dots t_{n-1})-t_n,\,\,1\leq j\leq n
\]
which transforms the hyperplane $t_n=0$ (corresponding to the original $\Rnn$) into the hyperplane 
\begin{equation}\label{sum0}
\sum_{j=1}^n y_j=0,
\end{equation}
and the perpendicular  line into the perpendicular line: $\varphi(0,\dots,0,1)=(-1,\dots,-1)$.
By (\ref{ell1}), the unit ball in the entropy norm $B(h^\a_\mu)$ is mapped into the intersection of the ball in $\Rn$ of $\ell^1$-norm of radius $2$ with the hyperplane (\ref{sum0}) which we denote by $H$.

 The Jacobian matrix of $\varphi$ is
 \[
 D=\begin{pmatrix}\log|\lambda_1(A_1)|&\log|\lambda_1(A_2)|&\cdots &\log|\lambda_1(A_{n-1})|& -1\\
 \log|\lambda_2(A_1)|&\log|\lambda_2(A_2)|&\cdots &\log|\lambda_2(A_{n-1})|& -1\\
 \cdots & \cdots &\cdots &\cdots\\
 \log|\lambda_n(A_1)|&\log|\lambda_n(A_2)|&\cdots &\log|\lambda_n(A_{n-1})& -1
 \end{pmatrix}.
 \]  
 
 By Proposition \ref{diag} we have $\la_i(A_j)=\varphi_i(\ep_j)$, so if $\gamma(Z(A))=U_K$, the matrix $D$ with the last column removed is the transpose of the matrix that appears in the definition of the {\em regulator}  $R_K$  of the number field $K$ \cite[Chapter 2, Section 4]{BS}. 

Notice that  
the determinants of all $(n-1)\times(n-1)$ sub-matrices of $D$ obtained by deleting the last column and any row are equal up to the sign. By (\ref{sumcolumnzero})
the rows of $D$ with the last column removed are linearly dependent and sum to zero. This allows us to change the matrix $D$ with the $i$-th row removed into the matrix $D$ with the $j$-th row removed through a series of row operations (with $1 \le i, j \le n$). Since row operations affect only the sign of the determinant of a matrix, we conclude that the absolute value of the Jacobian $\frac{\partial (y_1,\dots,y_n)}{\partial (t_1,\dots,t_n)}$ is equal to $nR$, where $R$ is the absolute value of
the determinant of any such  $(n-1)\times(n-1)$ sub-matrix, hence $dy_1\wedge\dots\wedge dy_n=nRdt_1\wedge\dots\wedge dt_n$.
We write $dt_1\wedge\dots\wedge dt_n=ds\wedge dt_n$, where $ds$ is the $(n-1)$-dimensional volume in the hyperplane $t_n=0$, and similarly, $dy_1\wedge\dots\wedge y_n=ds'\wedge dt$, where $ds'$ is the $(n-1)$-dimensional volume in the hyperplane $H$, and $dt$ is in the perpendicular direction. 
Taking into account that the unit vector $(0,\dots,0,1)$ is mapped to the vector $(-1,\dots,-1)$ of length $\sqrt{n}$, i.e. that $dt=\sqrt{n}dt_n$,
we conclude that $ds'=\sqrt{n}Rds$, and
for the $(n-1)$-dimensional volumes we have
\[
vol(H)=\int_{H}ds'=\sqrt{n}R\int_{B(h^\a_\mu)}ds=\sqrt{n}Rvol(B(h^\a_\mu)).
\]
Using the formula for the volume of the intersection of a unit ball in the $\ell^1$-norm in $\Rn$ by a hyperplane  passing through the  origin, see  \cite[Prop. II.7]{MP}, for the ball of radius $2$ we have:
\[
vol(H)=\sqrt{n}\binom{2n-2}{n-1}\frac1{(n-1)!}.
\]
Therefore $vol(B(h^\a_\mu))=\binom{2n-2}{n-1}\frac{1}{R(n-1)!}$, and
\begin{equation}\label{Fried-regulator}
h^*(\a)=\frac{2^{n-1}}{(n-1)!vol(B(h^\a_\mu))}=\frac{R2^{n-1}}{\binom{2n-2}{n-1}},
\end{equation}
where  
  \begin{equation}\label{regulator-index}
 R=kR_K
 \end{equation}
 and $k=[U_K:\gamma(\a)]\ge1$.
It thus follows  that the lower bounds for the Fried average entropy will be found when $\gamma(\a)=U_K$.

\begin{thm} \label{Fried-entropy-bound} The Fried average entropy $h^*(\a)$
 of  a Cartan action of a given rank $n-1\ge 2$ is bounded away from zero 
 by a positive function  that grows exponentially with $n$,  
 \begin{equation}\label{friedexp}
 h^*(\a)>0.000752\exp(0.244n),\end{equation}
 and, furthermore, 
 \begin{equation}\label{friedabs}h^*(\a)\geq 0.089.\end{equation}
  \end{thm} 
  \begin{proof}

We use Zimmert's analytic lower bound for regulators \cite{Z}. It states that for a totally real number field $[K:\Q]=n$, and any $s>0$
\begin{equation}\label{Zim}
R_K\geq a(s)\exp(b(s)n),
\end{equation}
where 
\[a(s)=(1+s)(1+2s)\exp\left(\frac2{s}+\frac1{1+s}\right)
\]
 and 
\[
b(s)=\log\left(\frac{\Gamma(1+s)}{2}\right)-(1+s)\frac{\Gamma'}{\Gamma}\left(\frac{1+s}{2}\right).
\]
Due to \eqref{Fried-regulator}, in order to obtain an exponential lower bound for the Fried average entropy we need  $b(s)>\log 2$, and for $s=0.35$ we obtain
\begin{equation}\label{Z035}
 R_K>0.000376 \exp(0.9371n). 
 \end{equation}

Using the estimate for the middle binomial coefficient
\begin{equation}\label{rough}
\frac{4^{n-1}}{n}\le \binom{2n-2}{n-1}\le 4^{n-1}
\end{equation}
we obtain
\[
\frac1{2^{n-1}}\le \frac{2^{n-1}}{\binom{2n-2}{n-1}}\le \frac{n}{2^{n-1}}.
\]

Using the bound (\ref{Z035}), we obtain
\begin{equation}\label{better}
h^*(\a)=\frac{R_K2^{n-1}}{\binom{2n-2}{n-1}}>\frac{0.000376\exp(0.9371n)}{2^{n-1}}>0.000752\exp(0.244n),
\end{equation}
thus proving the  first assertion of Theorem~\ref{Fried-entropy-bound}, \eqref{friedexp}. The
minimum of this lower bound (achieved for $n=3$) is equal to 
$0.001565\dots$ that is weaker than  the claimed lower bound \eqref{friedabs}.

Inspection of the number fields data at http://www.lmfdb.org/ identifies the quartic totally real number field of discriminant $725$ as
the field that minimizes $h^*(\a)$ among the fields covered by the tables. 
For it $h^*(\a)=0.330027...$,  and we conjecture  that this gives the value $h_{\text{min}}$ of   $h^*(\alpha)$  for all Cartan actions $\alpha$. 

Although we are not able to prove this statement (see Conjecture \ref{minimizing}), we can improve the lower bound to \eqref{friedabs}. Along the way we explain what number-theoretic tools are available for the actual proof that for any Cartan action $\a$ (i.e. for any totally real number field $K$) 
\[
h^*(\a)\geq h_{min},
\]
and what the limitations presently are.
\begin{lem} The Fried average entropy  $h^*(\a)$ of a
 Cartan action of  rank $n-1$ for $3\leq n\leq 7$ satisfies $h^*(\a)\geq h_{min}=0.330027...$.
\end{lem}
\begin{proof} 

Suppose  $K$ is a totally real number field of degree $n$ for which
\begin{equation}\label{min}
h^*(\a)<h_{min}.
\end{equation}
Then 
\[
\frac{0.000376\exp(0.9371n)2^{n-1}}{\binom{2n-2}{n-1}}<0.33002,
\]
and by a direct numerical calculation, we find that $n\leq 16$. 

Now we use Friedman's lower bound for the regulator of a totally real number field $[K:\Q]=n$ \cite{F},
\[
R_K>2g(1/D_K),
\]
where
\[
g(x):=\frac1{2^n4\pi i}\int_{2-i\infty}^{2+i\infty}(\pi^nx)^{-s/2}(2s-1)\Gamma(\frac{s}{2})^nds.
\]
It is known that $g(x)$ tends to $-\infty$ as $x\to 0+$, and that $g(x)$ is positive and vanishes exponentially fast for large $x$.
If  (\ref{min}) holds, then $R_K<\frac{0.33002\binom{2n-2}{n-1}}{2^{n-1}}$ which implies that for some $c_1(n)<c_2(n)$ 
(computable numerically) 
$D_K<c_1(n)$ or $D_K>c_2(n)$. In order to use the upper bound $c_1(n)$ we need absolute upper bounds for $D_K$ obtained by geometry of numbers \cite{PZ} that are $<c_2(n)$. For $n\leq 7$
such bounds are known, and the proof is completed by a finite check as follows. For $n=7$ the Friedman's upper bound for $D_K$ is smaller than the minimal discriminants of the totally real fields of this degree \cite{V} - so there are no totally real fields in degree $7$. 
For  degrees 3, 4, 5, 6 the Friedman's upper bounds for $D_K$ are $115, 2250, 40400$, and $710000$, respectively. The corresponding $19$ totally real number fields found in http://www.lmfdb.org/ ($2$ of degree $3$, $7$ of degree $4$, $4$ of degree $5$ and $6$ of degree $6$) are given in Tables 1-4. They show that the minimum of the Fried average entropy is achieved on the quartic field $K=\Q(\ve)$ with discriminant $725$ and the defining polynomial $x^4-x^3-3x^2+x+1$, where $\ve$ is a fundamental unit. 
Using Pari-GP we find that for this field 
the index $[\mathcal O_K:\mathbb Z[\ve]]=1$, and by (\ref{3.6}) we conclude that $\gamma(C(A))=\mathcal O_K$, and hence $\gamma(Z(A))=\mathcal U_K\cap\gamma(C(A))=\mathcal U_K$.

\begin{table}[h]
\caption{Totally real number fields of degree $3$ with $D_K<115$}
\begin{tabular}{ c | c | c | c }
\hline
$D_K$ & $f$ & $R_K$ & $h^*(\a)$\\
\hline
$49$ & $x^3-x^2-x +1$& $0.525454$ & $0.350303$ \\
$81$ & $x^3-3x-1$& $0.849287$& $0.566191$\\
\hline
\end{tabular}
\end{table}

\begin{table}[h]
\caption{Totally real number fields of degree $4$ with $D_K<2250$}
\begin{tabular}{ c | c | c | c }
\hline
$D_K$ & $f$ & $R_K$ & $h^*(\a)$\\
\hline
$725$ & $x^4-x^3 -3x^2 +x+1$& $0.825068$ & $0.330027$ \\
$1125$ & $x^4-x^3 -4x^2 +4x+1$& $1.165455$& $0.466182$\\
$1600$&$x^4-6x^2+4$&$1.542505$&$0.617002$\\
$1957$&$x^4-4x^2-x+1$&$1.918363$&$0.767345$\\
$2000$&$x^4-5x^2 +5$&$1.852810$&$0.741124$\\
$2048$&$x^4-4x^2 +2$&$2.441795$&$0.976718$\\
$2225$&$x^4-x^3-5x^2+2x +4$&$2.064511$&$0.825804$\\
\hline
\end{tabular}
\end{table}

\begin{table}[h]
\caption{Totally real number fields of degree $5$ with $D_K<40400$}
\begin{tabular}{ c | c | c | c }
\hline
$D_K$ & $f$ & $R_K$ & $h^*(\a)$\\
\hline
$14641$ & $x^5-x^4-4x^3+3x^2+3x-$& $1.635694$ & $0.373873$ \\
$24217$&$x^5-5x^3-x^2+3x+1$&$2.399421$&$0.548439$\\
$36497$&$x^5-x^4-3x^3+ 5x^2+ x-1$&$3.550657$&$0.811579$\\
$38569$&$x^5-5x^3 +4x-1$&$3.155437$&$0.721243$\\
\hline
\end{tabular}
\end{table}

\begin{table}[h]
\caption{Totally real number fields of degree $6$ with $D_K<710000$}
\begin{tabular}{ c | c | c | c }
\hline
$D_K$ & $f$ & $R_K$ & $h^*(\a)$\\
\hline
$300125$&$x^6-x^5-7x^4 +2x^3+7x^2-2x-1$&$3.277562$&$0.416198$\\
$371293$&$x^6 -x^5-x^4 +4x^3+6x^2-3x-1$&$3.774500$&$0.479302$\\
$434581$&$x^6-2x^5-4x^4+5x^3+4x^2 -2x-1$&$4.187943$&$0.531802$\\
$453789$&$x^6-x^5-6x^4+ 6x^3+8x^2-8x+ 1$&$4.399962$&$0.558725$\\
$592661$&$x^6-x^5-5x^4+ 4x^3+5x^2-2x-1$&$4.525483$&$0.574665$\\
$703493$&$x^6- 2x^5- 5x^4+11x^3+2x^2-9x+1$&$5.233524$&$0.664574$\\
\hline
\end{tabular}
\end{table}

\end{proof}

\begin{figure}[thb]
\vspace*{-3.2cm}
  \includegraphics[scale=0.37]{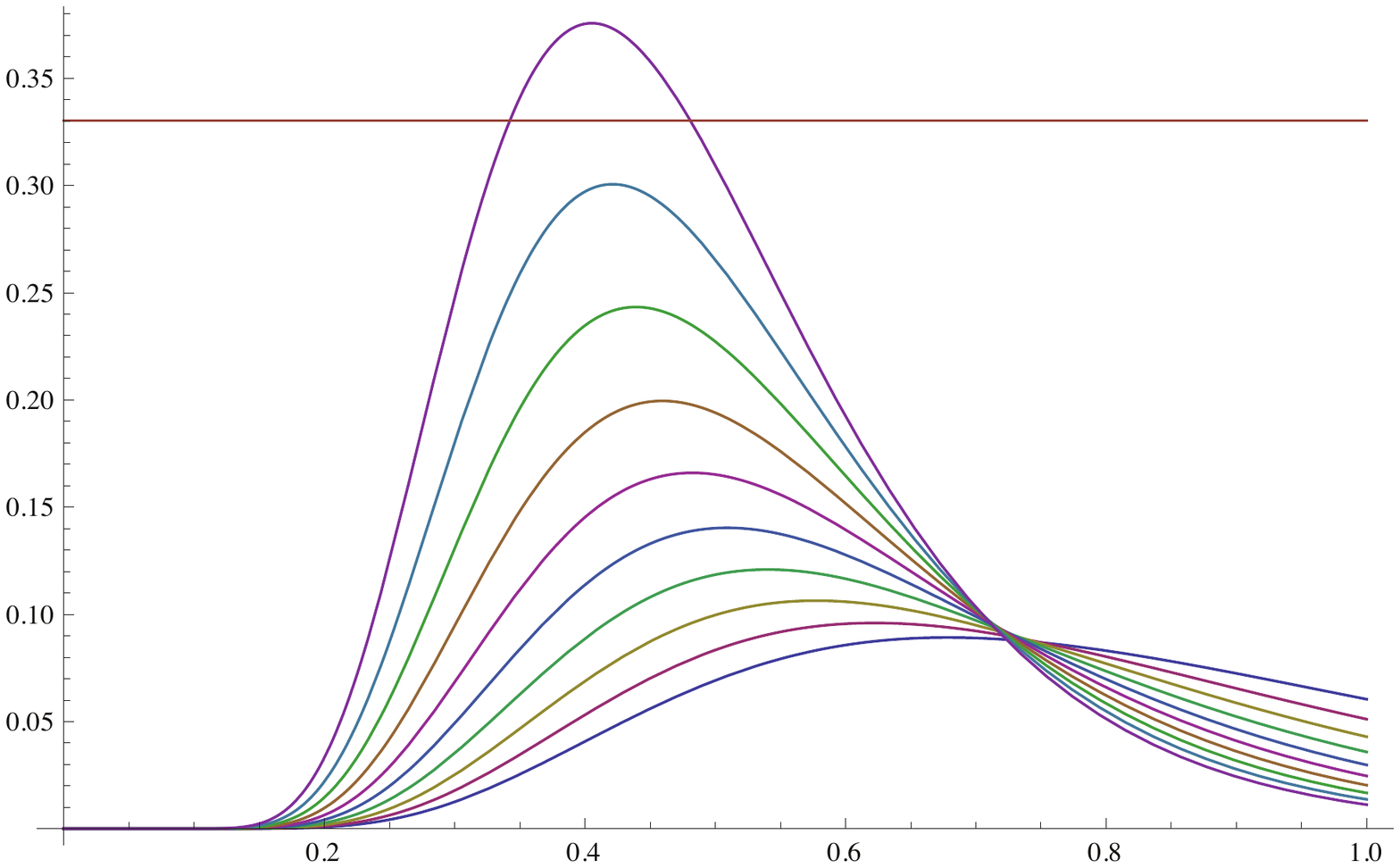}
\vspace*{-3.4cm}
   \caption{Plots of $Z(n,s)$ for $n=8,\dots,16,17$}
   \end{figure} 
Unfortunately, for $n>7$ such bounds are not known, but for $8\leq n\leq 16$   
we can improve this minimum of the lower bound (\ref{better}) by varying the parameter $s$ in the Zimmert analytical formula  (\ref{Zim}). 

Let $Z(n,s)=\frac{a(s)\exp(b(s))2^{n-1}}{\binom{2n-2}{n-1}}$ be the lower bound for $h^*(\a)$ obtained from the lower bound for the regulator (\ref{Zim}). 
Using numerical calculation, we obtain
\[
h^*(\a)\geq \min_n(\max_s Z(n,s))=0.089.
\]

The functions $Z(n,s)$ for $n=8,\dots,16,17$ are shown in Figure 2 (obtained by Mathematica). It is also evident from these plots that 
$n\leq 16$ is the best bound that can be obtained from Zimmert's analytic formula. 
\end{proof}

\begin{conj}\label{minimizing} The Cartan action $\a$ corresponding to the  quartic totally real number field of discriminant $725$ and the defining polynomial $x^4-x^3-3x^2+x+1$ minimizes the Fried average entropy $h^*(\a)$ among all Cartan actions $\a$. For that action $h^*(\a)=0.330027...=h_{min}$.
\end{conj}

Proposition~\ref{mahler-exp} and Theorem \ref{Fried-entropy-bound} make it plausible that $\ell$-entropies of Cartan actions
grow exponentially with rank for all values of $\ell$ proportional to the rank. Since  for small coefficients of proportionality Proposition~\ref{mahler-exp} applies we need only to consider coefficients larger than $ec$. 

\begin{conj}\label{conj-exp} Let $\ell\ge \frac{e}{2}\log(\frac{1+\sqrt 5}{2})n$. There exist positive constants $a$ and $b$ such that
for any Cartan action $\a$ of rank $n-1$
\[h^\ell(\a)>a\exp bn.
\]
\end{conj} 

\subsection{Rigidity of maximal rank actions} The reason we pay such attention to Cartan actions is that they provide essentially the universal model for positive entropy maximal rank actions of
 $\Zk, k\ge 2$. This is one of the principal results of \cite{KRH-a}. 
\begin{theorem*}\label{TMain1}  Let $\a$ be a $C^r, \,1+\theta\le r\le  \infty $ maximal rank positive entropy action 
  on a smooth manifold $M$  of dimension $n\ge 3$.

Then there  exist:
\begin{itemize}

\item  disjoint measurable sets of equal measure  $R_1,\dots, R_m\subset M$ such that $R= \bigcup_{i=1}^m R_i $ has full measure and the action $\a$  cyclically interchanges those sets. Let $\Gamma\subset \Znn$ be the stationary subgroup of any of the sets $R_i$ ($\Gamma$ is of course isomorphic to $\Znn$);
 
\item  a Cartan action  $\ao$ of $\Gamma$   by affine transformations   of either the torus $\Tn$ or the infratorus  $\Tn/\{\pm\Id\}$ that we will call  {\em the algebraic model}; 
  
 \item measurable maps  $h_i: R_i\to\Tn$ or $h_i: R_i\to\Tn/\{\pm\Id\}, \, i=1,\dots, m$; 
 \end{itemize}
 such that
  
\begin{enumerate}
\item \label{TM1} $h_i$ is bijective almost everywhere and $(h_i)_*\mu=\l$, the Lebesgue (Haar) measure on $\Tn$ (correspondingly $\Tn/\{\pm\Id\}$);

\item \label{TM2} $\ao\circ h_i=h_i\circ \a$;

\item \label{TM3} for almost every $x\in M$  and every ${\bf n}\in \Znn$ the restriction of $h_i$ to the 
stable manifold $W^s_x$ of $x$ with respect to $\a({\bf n})$ is  a $C^{r-\epsilon} $ diffeomorphism for any $\epsilon>0$.

\item \label{TM4} $h_i$ is  $C^{r-\epsilon}$ in the sense of Whitney on a set  whose complement  to $R_i$ has arbitrary small measure;  they are saturated by  local stable manifolds.
\end{enumerate}
\end{theorem*}

Notice that taking a finite factor does not change the entropy function and hence the value of Fried average entropy. Passing to a finite index subgroup multiplies the value of the Fried average entropy by the index of the subgroup, see Proposition~\ref{Fried-properties}, \eqref{entropy-index}. Since for a weakly mixing action $m=1$, Theorem~A, 
 Proposition~\ref{general-position} and Theorem~\ref{Fried-entropy-bound} imply the following   rigidity property of the Fried average  entropy. 

\begin{cor}\label{CorollaryMain} The Fried average entropy of a weakly mixing maximal  rank action  of $\Zk, \, k\ge 2$,  is either equal to zero or is greater than $0.089$; furthermore the lower bound grows with $k$ exponentially, as in \eqref{friedexp}. 
\end{cor}
From (\ref{Fried-regulator}) and taking into account that for a general ergodic action the Fried average entropy may only be a rational multiple of that for the weakly mixing ergodic components of its restriction to a certain finite index subgroup, we also have
\begin{cor} The Fried average entropy of a maximal  rank action  of $\Zk, \, k\ge 2$, may take only countably many values.
\end{cor}

\section{Connection between the  Fried average entropy and slow entropy}
\subsection{Refining the slow entropy type invariants}\label{SectionSlow}
Slow entropies  for actions of higher rank abelian groups have been defined and discussed in \cite{KTh}. As we mentioned in the introduction, the motivation for using  slow entropies for  smooth actions of higher rank abelian groups is as follows: the number of different  codes  with respect to 
a nice partition for any smooth action grows no faster than exponentially  against the size of the  ball in  a group, i.e. the  logarithm of this number grows at most linearly. Hence 
in order to produce a non-vanishing invariant  logarithms of those numbers need to be normalized not by the number of elements in the ball but by  its $k$th root where, in the case of abelian groups, $k$ is the rank of the acting  group.  This corresponds to the slow entropy with the scale function $s^{1/k}$.

One needs to  elaborate the  constructions from  \cite[Section 1.3]{KTh}. 
First, we will restrict ourselves to the case of $\Zk$ actions that are the only one of interest in the context of this paper. In this context the F\o lner sequence that appears in the discussion  will be a sequence of balls in a  metric in $\Zk$ induced by a norm in the ambient space $\Rk$.  Neither of two
 specific suggestions  from  \cite[Section 1.3]{KTh} produce numerical invariants that distinguish maximal rank full entropy actions. Other choices are alluded to in the sentence: ``The reader will find without much difficulty other convenient characteristics of the asymptotic growth''. We will describe one such choice shortly. This will produce a numerical invariant that depends on a choice of a norm in $\Rk$ and in particular,   depends on a choice of generators  for a $\Zk$ action. To address this,  we will minimize the  value of this  norm-dependent  slow entropy over all choices of norm with the unit ball of volume one. We will show that in the case we are interested, namely 
 the case ({\bf P}) of the alternative in  section~\ref{SectionFried} for affine actions on the torus and Lebesgue measure   the minimum is always positive and is achieved for a certain  polyhedral norm. By extending this definition to $\Rk$ actions and connecting it with the quantities that appear in the Ledrappier-Young formula for entropy  one can show that the resulting notion satisfies conditions (i)-(iv) from the introduction  and that  the statement about polyhedral minimizers also holds in general.  Detailed treatment of the general case  will appear in  a forthcoming paper by Changguang Dong.
 
In the present paper we will restrict  specific  calculations to the case of Cartan actions. We will show that in each dimension the values of the Fried average entropy and the slow entropy determine each other, that via Theorem~A will imply corresponding results for arbitrary maximal rank actions. 
 
For a moment we return to a general situation. Let  $\a$ be an action of  a discrete group $\Gamma$ by transformation of  a space $X$ preserving measure $\mu$. 
All slow entropy constructions begin with defining the essential  number of substantially different orbit segments 
 or blocks. A block is defined by an initial condition $x\in X$ and a finite set $F\subset \Gamma$  it is simply $\{ \a(\gamma)x, \gamma\in F\}$. The word ``essential''  means that sets of  fixed small measure can be ignored. ``Substantially different''
 is defined by using a family of metrics or semi-metrics  $d_F$  in the spaces of $F$-blocks and considering blocks with $d_F$
 distances greater  that a fixed small number  as  essentially different.  There are  two standard ways to define $d_F$.
 One, that is based on an idea of coding. One fixes a finite  partition $\xi$ of the space $X$ into measurable sets $c_1,\dots,c_N$ and defines $d_F$ as the pull-back of the Hamming metric in the space of codes, see \cite[Section 1.1]{KTh} for details. In the  other method one  assumes that $\a$ is an action by homeomorphisms of a compact metrizable space, starts  with a metric $d$ in $X$ and defines $d_F=\max_{\gamma\in F}d\circ\a(\gamma)$.
 Then one defines $S_{d_F}(\a, \epsilon, \delta)$ as the minimal number of $d_F$ balls whose union has  measure
 $\ge 1-\delta$. To make notations more suggestive this quantity in the first of the above cases is denoted by 
 $S_\xi^H(\a, F,\epsilon, \delta)$ and in the second by $S_d(\a,F,\epsilon, \delta)$.
 
Now we go back to the abelian case. Let $\a$ be an action of $\Zk$. Given a norm $p$ on $\mathbb{R}^k$ let $F^p_s$ be the intersection of the lattice $\Zk$ with  the ball centered in $0$ and radius $s$. We define the slow entropy of $\alpha$ with respect to the norm $p$ and the partition $\xi$ as 

\begin{equation}\label{eqslowpartition}sh(p,\alpha,\xi)=\lim_{\epsilon,\delta\to 0}\limsup_{s\to\infty} \frac{1}{s}\log S^H_{\xi}(\alpha,F_s^p,\epsilon,\delta).\end{equation}
Then we define 

\begin{equation}\label{eqslowent}sh(p,\alpha)=\sup_{\xi}sh(p,\alpha,\xi).\end{equation}

\cite[Proposition 1]{KTh} states that slow entropy can be defined using any generating sequence of  partition, e.g.  in the case of an action by homeomorphisms any sequence of partitions such that the maximal diameter on an element  converges to zero. \cite[Proposition 2]{KTh} then implies that 

\begin{equation}\label{eqslowmetric}sh(p,\alpha)=\lim_{\epsilon,\delta\to 0}\limsup_{s\to\infty} \frac{1}{s}\log S_d(\alpha,F_s^p,\epsilon,\delta).\end{equation}

Notice that since \eqref{eqslowent} is a measure-theoretic isomorphism invariant the limit in \eqref{eqslowmetric} does not depend on the choice of metric $d$. 
 
Finally we define the slow entropy of $\alpha$ as $$sh(\alpha)=\inf_{p:vol(p)=1}sh(p,\alpha)$$ where $vol(p)$ is the volume of the unit ball in the norm $p$ (we are assuming $\Zk\subset\Rk$ has co-volume $1$). If $\alpha$ is understood, we will omit reference to it like $sh(p)=sh(p,\alpha)$. 

\begin{rem} Taking the logarithm in \eqref{eqslowpartition} is essential  for obtaining a numerical invariant  that takes finite positive values in many cases. This is not true for two suggestions for asymptotic invariants in \cite[Section 1.3]{KTh}. The first of those may be too subtle and the second is too crude. A good analogy between the latter and our definition  is between the Hausdorff dimension of a set  and  Hausdorff measure  corresponding  to a fixed value of Hausdorff dimension.
\end{rem}

\subsection{Slow entropy and Lyapunov exponents}

Let $\a$ be a $\Zk$ action by affine maps of  $\mathbb T^n$. Let $\Rn=E_1\oplus\dots\oplus E_r$ be the splitting into coarse Lyapunov subspaces of this linear action and let us fix once and for all the max norm $|\cdot|$ on $\Rn$ with respect to the splitting $E_1\oplus\dots\oplus E_r$ i.e. if $v=v_1+\dots+v_r$ with respect to the splitting $E_1\oplus\dots\oplus E_r$ then $|v|=\max_i|v_i|$. 

Let $p$ be a norm on $\Rk$ such that the unit ball has volume one. We want to compute $sh(p,\alpha)$.  We will use \eqref{eqslowmetric}  for the metric on the torus 
defined by the above defined norm $|\cdot |$. Thus the  metric $d_{F^p_s}$ is defined by the norm 
 $$|v|^p_s=\max_{t:p(t)\leq s}|\alpha(t)v|.$$
Let $B_s^p$ be the unit ball in the norm $|\cdot|^p_s$.
\begin{lem} \begin{equation}\label{pesinlinearvolume}sh(p,\alpha)=\lim_{s\to\infty}-\frac{1}{s}\log vol (B_s^p)\end{equation}
\end{lem}
\begin{proof} Due to translation invariance of the metric $d_{F^p_s}$ all balls of  a given radius have the same volume that for a small enough $\epsilon$ and all $s$ is equal to $\epsilon^n vol (B_s^p)$.  Hence
$$S_d(\alpha,F_s^p,\epsilon,\delta)\ge\frac{1-\delta}{\epsilon^n vol (B_s^p)}$$
that gives  a lower bound independent of $\epsilon$ and $\delta$. 

 To get an upper bound notice that the parallelepiped $\epsilon B_s^p$ tiles $\Rn$and hence the tiles provide a disjoint cover  of a fundamental domain
for $\Tn$ except for a neighborhood of the boundary of the size $\epsilon\,diam(B_s^p)$. In particular, for a small enough $\epsilon$ the tiles cover an area of volume greater then one half. 
\end{proof}
Now, $v\in B_s^p$ if and only if $|v|^p_s=\max_{t:p(t)\leq s}|\alpha(t)v|\leq 1$ and $vol(B_s^p)$ is up to a constant  the product of the corresponding $vol(B_s^p\cap E_i)$.  Let $J_i(t)$ be the Jacobian of $\a(t)$ along $E_i$ and let $\hat\chi_i(t)=\log |J_i(t)|$. $\hat\chi_i$ is proportional to any Lyapunov exponent appearing in  the coarse Lyapunov direction $E_i$. 

Let $$M_i^p(s)=\max_{t:p(t)\leq s} \hat\chi_i(t).$$ Then we have that (up to a constant) 
\[
vol(B^p_s)=\prod_i\exp(-M_i^p(s)).
\]
Observe that since $\hat\chi_i$ and $p$ are $1-$homogeneous we get that $M_i^p(s)=sM^p_i(1)$ hence $$-\frac{1}{s}\log vol(B^p_s)=\sum_iM_i^p(1)=\sum_i\max_{t:p(t)\leq 1}\hat\chi_i(t).$$

Thus for any action $\a$ by affine maps of a torus  the slow entropy for Lebesgue measure  with respect to  the norm $p$ is  equal to 

\begin{equation}\label{slow-linear}sh(p,\alpha)=\sum_i\max_{t:p(t)\leq 1}\hat\chi_i(t)=\sum_i\max_{t:p(t)\leq 1}\chi_i(t)=\sum_ip^*(\hat\chi_i)\end{equation}
where $p^*$ is the dual norm to $p$, i.e. $p^*(u)=\max_{p(t)=1}|u(t)|$.

Notice that since by the Pesin entropy formula $h(\a(t))=\frac{1}{2}\sum_i|\chi_i(t)|$ (see \eqref{ell1}) and hence
\begin{equation}\label{slow-Pesin}sh(p,\alpha)\ge 2\max_{t:p(t)\leq 1}h(\a(t)).
\end{equation}

Now we want to minimize this over all norms $p$ with unit ball of volume $1$. 

First, notice that there is dichotomy  that already appeared in section~\ref{SectionFried}. Namely either all elements of the suspension, save for identity, have positive entropy, or a non-trivial element of the suspension has zero entropy. Since we are dealing with actions preserving Lebesgue measure,  by the Pesin entropy formula the second alternative is equivalent to non-triviality of the intersection of all Lyapunov hyperplanes, i.e. existence of a common annihilator  for all Lyapunov exponents. First we show, that, similarly to the case of Fried average entropy, this dichotomy is equivalent to the zero-positive  dichotomy for the slow entropy. 

\begin{prop} For an affine action $\a$ on a torus $sh(\a)=0$ if and only if intersection of all Lyapunov hyperplanes
contains a non-zero vector, or, equivalently, the entropy function is not a norm, or, equivalently, Fried entropy vanishes.   
\end{prop}

\begin{proof} First, let's assume that the entropy function is a norm. This implies that the entropy reaches a positive minimum $m$ on the unit sphere. Consider the unit ball  of volume  one  with respect to a  norm $p$. It  is a convex centrally symmetric body  and hence contains a vector $t$ of length greater than a certain absolute constant $C$.
Hence  by \eqref{slow-Pesin} $sh(p,\alpha)\ge 2mC$.

Now assume that for some  $t\neq 0,\,\ \chi_i(t)=0$ for all $i$. Fix an $\epsilon>0$ and consider a norm $p_\epsilon$ whose unit ball is the product of a segment of the line  generated by $t$ and the ball of radius $\epsilon$ in the  hyperplane perpendicular to $t$.  Notice that all Lyapunov exponents of $\a$ are bounded  in absolute value on the unit ball of the norm $p_\epsilon$ by  $\epsilon\cdot const$.
\end{proof}

 Thus we only need to consider the case of positive slow entropy. It appears exactly when the Lyapunov exponents (and hence the vectors $\hat \chi_i$) generate the dual space to $\Rn$. 

\subsection{Convex bodies and minimizers for affine actions} 
Let $\xi_1,\dots, \xi_n\in\Rk$ be a set of spanning  vectors and let $C$ be the symmetric convex hull of the $\xi_i's$, i.e. $C$ is the convex hull of the set $\{\xi_1,\dots, \xi_n, -\xi_1,\dots, -\xi_n\}$.  Let $p_C$ be the Minkowski functional associated to $C$. Remember that given a norm $p$, we denote  its dual norm by  $p^*$, and that $(p^*)^*=p$. Given a convex symmetric body, $C$ let $$C^\circ=\{x:| x\cdot\xi|\leq 1\;\;\;\mbox{for any}\;\;\;\xi\in C\}.$$ Observe that $p_C^*=p_{C^\circ}$.

\begin{lem}
$$p^*_C(v)=\max_i |\xi_i\cdot v|$$
\end{lem}
\begin{proof} We claim that $$C^\circ=\bigcap_i\{x:|x\cdot\xi_i|\leq 1\}.$$ Indeed, let $A=\bigcap_i\{x:|x\cdot\xi_i|\leq 1\}$. By definition of $C^\circ$ we have that $C^\circ\subset A$. We shall use the following simple property:
\medskip

If $C^\circ\subset A$ then $A^\circ\subset C$ and $C^\circ=A$ if and only if $A^\circ=C$.
\medskip

%\end{lem}
\noindent We have then that $A^\circ\subset C$, let us prove that $C\subset A^\circ$. Let $\xi\in C$. Then $$\xi=\sum_it_i^+\xi_i-\sum_it_i^-\xi_i=\sum_i(t_i^+-t_i^-)\xi_i,$$ where $\sum_it_i^++\sum_it_i^-=1$, $t_i^\pm\geq 0$.

Take $x\in A$, then $|x\cdot\xi_i|\leq 1$ for any $i$ and hence 
\begin{eqnarray*}
|x\cdot \xi|&=&\left|x\cdot\sum_i(t_i^+-t_i^-)\xi_i\right|=\left|\sum_ix\cdot(t_i^+-t_i^-)\xi_i\right|\\&\leq& \sum_i|t_i^+-t_i^-||x\cdot\xi_i|\leq\sum_it_i^++t_i^-=1
\end{eqnarray*}

So, we get that $\xi\in A^\circ$ and hence $$C^\circ=\bigcap_i\{x:|x\cdot\xi_i|\leq 1\}.$$ Now, for this $C^\circ$ the Minkowski functional is $$p_{C^\circ}(v)=p^*_C(v)=\max_i |\xi_i\cdot v|.$$ \end{proof}
Given a norm, we denote by $vol(p)$ the volume of the unit ball in the norm $p$.

Instead of minimizing the quantity $$\sum_ip^*(\xi_i)$$ over the norms  $p$ with  $vol(p)=1$ we shall minimize the following more homogeneous quantity. Given a family of functionals $ \Xi =(\xi_1,\dots,\xi_n)$ and a norm $p$ in $\Rk$ we define

$$SH_{\Xi}(p)=\frac{\sum_ip^*(\xi_i)}{vol(p)^{\frac{1}{k}}}.$$

This last quantity is invariant by $c\mapsto cp$, hence this is the same.

\begin{prop}
Given  a spanning set of functionals  $\Xi=(\xi_1,\dots, \xi_n)$, the minimum of $SH_{\Xi}(p)$ is attained at a norm of the form $q(x)=\max_ic_i|\xi_i\cdot x|$ for some values of  $c_i\geq 0$. 
\end{prop}

\begin{proof}
Let $p$ be a norm. Let us take $C$ the symmetric convex hull of $$\left\{\frac{\xi_i}{p^*(\xi_i)}\right\}_i.$$ Let $q^*$ be the Minkowski functional of $C$ and $q$ its dual norm.  Observe that $C\subset B^{p^*}(1)$ (the unit ball of $p^*$). By definition we have that $$q^*(\xi_i)\leq p^*(\xi_i)$$ for every $i$. Also we have that the unit ball of $q$ is $C^\circ$. And hence $B^p(1)=(B^{p^*}(1))^\circ \subset  C^\circ=B^q(1)$ and hence $$vol(p)\leq vol(q).$$ Hence we get that $$SH_\Xi(q)\leq SH_\Xi(p).$$ Since $C$ is a convex polyhedron the norm  $q$ has the  form
\begin{equation}\label{norms-preferred}\max_ic_i|\xi_i\cdot x|.
\end{equation}  Now by homogeneity  we may without loss of generality assume that  $\sum_ic_i=1$ and consider  the  function $SH_\Xi$ on the  unit simplex. Notice that $SH_\Xi$ is defined if  and only if  \eqref{norms-preferred} is a norm. This is true if all coefficients $c_1,\,\,i=1,\dots,n$ are positive, i.e in the interior of the simplex,   as well as for those  points on the boundary where the  vectors $\x_i$ with non-zero coefficients form a spanning set. The function  $SH_\Xi$ is continuous at all points of the simplex where it is defined and goes to infinity elsewhere.  Hence it reaches a positive  minimum. 
\end{proof}

Next we will compute how $SH_\Xi$ behaves under linear change of variable.
\begin{lem}\label{trans2}
Given an invertible linear map $L$ on $\Rk$, a norm $q$ on $\Rk$ and linear functionals $\Xi=(\xi_1,\dots, \xi_n)$  in $\Rk$, let $\eta_i=(L^t)^{-1}\xi_i,\,\,E=(\eta_1,\dots,\eta_n)$ and $L_*q(x)=\hat q(x)=q(Lx)$. Then $$SH_\Xi(L_*q)=SH_\Xi(\hat q)=SH_E(q)|\det L|^{1/k}.$$
\end{lem}

\begin{proof} Let $\hat q(x)=q(Lx).$ 
Then $\hat q^*(\xi)=q^*((L^t)^{-1}\xi)$ and $L(B^{\hat q}(1))=B^q(1).$ Hence $$\det(L)vol(\hat q)=vol(q).$$ \end{proof}

\subsection{Slow entropy for Cartan actions}\label{sbsSlowCartan} 
\begin{thm}\label{slow-entropy-bounds}
Let $\alpha$ be a Cartan action on the torus $\mathbb{T}^n$ then
\begin{equation}\label{slow-regulator} sh(\a)= C(n)R^{\frac{1}{n-1}}\end{equation}
 where $R=kR_K$ is as in (\ref{regulator-index}), i.e. $R_K$ is the regulator and $k=[U_K:\gamma(\a)]$, and 
 \begin{equation}\label{slow-entropy-estimate}\frac{n-1}{2}\le C(n)\le n-1.\end{equation} 
\end{thm}

\begin{proof}
Let  $\chi_i$, $i=1,\dots, n$ be linear functionals in $\mathbb{R}^{n-1}$ corresponding to the Lyapunov exponents of a Cartan action $\alpha$. Hence $\sum_{i=1}^n\chi_i=0$ and $\chi_1,\dots\chi_{n-1}$ are linearly independent. 

Let $\hat q(x)=\max_ic_i|\chi_i\cdot x|$ for some $c_i\geq 0$. Notice that in order for $\hat q$ to be a norm at most one coefficient $c_i$  may vanish. Thus, without loss of generality, we may assume that $c_i>0$ for $i<n$. Now let $L$ be a linear transformation such that $(L^t)^{-1}(c_i\chi_i)=e_i$, $i=1,\dots, n-1$, where $e_i$ is the canonical basis in $\mathbb{R}^{n-1}$. Then we have by Lemma \ref{trans2} that 
$$SH_{\chi_1,\dots,\chi_n}(\hat q)=SH_{\frac{e_1}{c_1},,\dots, \frac{e_{n-1}}{c_{n-1}},\eta}(q)|\det L|^{1/k}$$
 where $\eta=(L^t)^{-1}\chi_n=-\sum_{i=1}^{n-1}\frac{e_i}{c_i}$ and 
\begin{eqnarray*}
q(x)&=&\hat q(L^{-1}x)=\max_ic_i|\chi_i\cdot L^{-1}x|=\max_ic_i|(L^t)^{-1}\chi_i\cdot x|\\&=&\max(\max_{1\leq i\leq n-1}c_i|\frac{e_i}c_i\cdot x|, c_n|(L^t)^{-1}\chi_n\cdot x|)\\&=&\max(\max_{1\leq i\leq n-1}|x_i|, c_n|\eta\cdot x|).
\end{eqnarray*}

Finally, 
\begin{equation}\label{minimizer-regulator}
\begin{aligned}
SH_{\chi_1,\dots,\chi_n}(\hat q)=&SH_{\frac{e_1}{c_1},,\dots, \frac{e_{n-1}}{c_{n-1}},\eta}(q)|\det L|^{\frac{1}{n-1}}\\
=&\frac{\sum_{i=1}^{n-1}\frac{q^*(e_i)}{c_i}+q^*(\eta)}{vol(q)^{\frac{1}{n-1}}}\left(\prod_{i=1}^{n-1}c_i\right)^{\frac{1}{n-1}}R^{\frac{1}{n-1}}
\end{aligned}
\end{equation}
 where $R=|\det(\chi_i)_{1\leq i\leq n-1}|$ as in (\ref{regulator-index}) is  the product of the regulator $R_K$ and $k=[U_K:\gamma(\a)]$.
 
 At this point we already see that values of $c_1,\dots,c_n$ for which  the minimum of $SH_{\chi_1,\dots,\chi_n}(\hat q)$ is reached  do not depend on $\chi_1,\dots,\chi_n$. Hence \eqref{minimizer-regulator} implies \eqref{slow-regulator}. 

We can assume without loss of generality that $c_n\leq c_i$ for every $i$, and let $\hat\eta=c_n\eta$. Then we have that $\hat\eta=\left(-\frac{c_n}{c_1},\dots,-\frac{c_n}{c_{n-1} }\right)$ and all its coordinares are of absolute value smaller than or equal to $1$.
Thus $$q(x)=\max(\max_{1\leq i\leq n-1}|x_i|), |\hat\eta\cdot x|$$ and we need to find upper and lower bounds for 
the minimum of
\[
\frac{\sum_{i=1}^{n-1}\frac{q^*(e_i)}{c_i}+\frac{q^*(\hat\eta)}{c_n}}{vol(q)^{\frac{1}{n-1}}}\left(\prod_{i=1}^{n-1}c_i\right)^{\frac{1}{n-1}}
\]
over $(c_1,\dots,c_n)$; more specifically, in order to prove \eqref{slow-entropy-estimate} we need to show that the minimum in question is  bound between $\frac{n-1}{2}$ and $n-1$.

\begin{lem}
$q^*(e_i)=1$ for  $i=1,\dots,n-1$.
\end{lem}

\begin{proof}
$$q^*(e_i)=\max_{q(x)\leq 1}|x_i|.$$ By the definition of $q(x)$ we get that if $q(x)\leq 1$ then $|x_i|\leq 1$; moreover, taking $x=e_i$ we get that $|x_i|=1$, $x_j=0$ for $j\neq i$ and $|\hat \eta\cdot x|=\frac{c_n}{c_i}\leq 1$. Hence $q(e_i)=1$ and $q^*(e_i)=1$. 
\end{proof}

\begin{lem}
$q^*(\hat \eta)=\min\{1, \sum_{i=1}^{n-1}\frac{c_n}{c_i}\}$. 
\end{lem}
\begin{proof}
Clearly $q^*(\hat\eta)\leq 1$. Let $\bar 1=(1,1,\dots, 1)$ be the vector with all coordinates equal to $1$, i.e. $\bar 1=\sum e_i$.  Then $|\hat\eta\cdot\bar 1|=\sum_{i=1}^{n-1}\frac{c_n}{c_i}$, hence $q(\bar 1)=\max\{1,|\hat\eta\cdot\bar 1|\}$. So, if $q(\bar 1)=1$ then $1\geq|\eta\cdot\bar 1|$  and  the claim follows. If, on the other hand, $q(\bar 1)=|\eta\cdot\bar 1|$, then $|\eta\cdot\bar 1|\geq 1$. 
\end{proof}

So there are two cases:

{\bf Case I:} $q^*(\hat\eta)=\sum_{i=1}^{n-1}\frac{c_n}{c_i}$.

Here  $|\hat\eta\cdot\bar 1|\leq 1$. Hence, if $|x_i|\leq 1$ for any $i$ then $|\hat\eta\cdot x|\leq|\hat\eta\cdot\bar 1|\leq 1$ and consequently  $q$ is the $l^{\infty}$ norm, 
$$q(x)=\max_i|x_i|.$$ 
In particular,  $vol(q)=2^{n-1}$, hence 
 $$\frac{\sum_{i=1}^{n-1}\frac{q^*(e_i)}{c_i}+\frac{q^*(\hat\eta)}{c_n}}{vol(q)^{1/k}}\left(\prod_{i=1}^{n-1}c_i\right)^{\frac{1}{n-1}}=\sum_{i=1}^{n-1}\frac{1}{c_i}\left(\prod_{i=1}^{n-1}c_i\right)^{\frac{1}{n-1}}.$$

The right hand side is always larger than or equal to $(n-1)$ and it minimizes precisely when $c_i=c$ for every $i=1,\dots, n-1$ and the value of $c_n$ is arbitrary as long as $c_n\leq \frac{c}{n-1}$.  This gives the above estimate in \eqref{slow-entropy-estimate}  and hence upper bound for $sh(\alpha)$.

{\bf Case II:} $q^*(\hat\eta)=1$

Here  $|\hat\eta\cdot\bar 1|\geq 1$. So the unit ball for $q$ is the unit cube intersected with the slice $|\hat\eta\cdot x|\leq 1$ and 
$$\frac{\sum_{i=1}^{n-1}\frac{q^*(e_i)}{c_i}+\frac{q^*(\hat\eta)}{c_n}}{vol(q)^{\frac{1}{n-1}}}\left(\prod_{i=1}^{n-1}c_i\right)^{\frac{1}{n-1}}=\frac{\sum_{i=1}^{n}\frac{1}{c_i}}{vol(q)^{\frac{1}{n-1}}}\left(\prod_{i=1}^{n-1}c_i\right)^{\frac{1}{n-1}}.$$

Now we have that $$\frac{\sum_{i=1}^{n}\frac{1}{c_i}}{vol(q)^{\frac{1}{n-1}}}\left(\prod_{i=1}^{n-1}c_i\right)^{\frac{1}{n-1}}\geq\frac{\sum_{i=1}^{n-1}\frac{1}{c_i}}{vol(q)^{\frac{1}{n-1}}}\left(\prod_{i=1}^{n-1}c_i\right)^{\frac{1}{n-1}}\geq \frac{n-1}{vol(q)^{\frac{1}{n-1}}}\geq \frac{n-1}{2}.$$
The last inequality follows since the unit ball of $q$ is inside the unit cube. This gives the below estimate in \eqref{slow-entropy-estimate} and hence the lower bound for $sh(\alpha)$.
\end{proof}

In light of good upper and lower bounds \eqref{slow-entropy-estimate} for the value of slow entropy it looks like  finding the actual minimizer and hence exact value of the constant $C(n)$ in Theorem~\ref{slow-entropy-bounds} will not much  improve our understanding  of the slow entropy and its connection with the Fried average entropy. Still  it would be nice to know the minimizer and  exact value of the constant $C(n)$ for aesthetic  reasons. Evidence from the low-dimensional cases as well as symmetry arguments support the assertion that the minimizer appears for $c_1=\dots=c_n$.  

\begin{conj}The minimum in \eqref{minimizer-regulator}  is attained at the norm  $q(x)=\max_i|\chi_i\cdot x|$. 
\end{conj}

Similarly to  the derivation of Corollary~\ref{CorollaryMain}, we obtain from Theorem A,  formula \eqref{Fried-regulator}, Theorem~\ref{Fried-entropy-bound}  and Theorem~\ref{slow-entropy-bounds}:

\begin{cor}\label{Corollary-Main2} The slow  $sh(\a)$ entropy of a weakly mixing maximal rank action $\a$  of $\Zk,\,k\ge 2$,
is expressed through its Fried average entropy $h^*(\a)$ as follows
$$sh(\a)=c(k)\binom{2k}{k}^{\frac{1}{k}}(h^*(\a))^{\frac{1}{k}},$$
where $k/4\le c(k)\le k/2$. In particular, the slow entropy is either equal to zero or uniformly bounded away from zero; the lower bound grows linearly with the rank. 
\end{cor}

A slight variation appears in the case of ergodic but not weakly mixing actions because of the normalization that is fully determined by the periodic part of the action but may not be uniform.

\end{document}